\documentclass[a4paper,11pt,titlepage,twoside]{article}

\usepackage{graphicx}
\usepackage[T1]{fontenc} 
\usepackage[utf8]{inputenc}
\usepackage[english]{babel}
\usepackage{soul}
\usepackage{amsfonts}
\usepackage{amsmath}
\usepackage{amsthm}
\usepackage{amssymb}
\usepackage{mathrsfs}
\usepackage[top=5cm, bottom=3cm, left=3.2cm, right=3.2cm]{geometry}
\usepackage{setspace}
\usepackage{afterpage}
\usepackage{extarrows}
\usepackage{fancyhdr}
\usepackage{titlesec}
\usepackage{enumitem} \setlist{nosep}
\usepackage[pdftex,breaklinks,colorlinks,linkcolor=blue,
anchorcolor=blue]{hyperref}
\usepackage{subcaption}

\theoremstyle{definition}
\newtheorem{defin}{Definition}[section]
\theoremstyle{plain}
\newtheorem{thm}[defin]{Theorem}

\newtheorem{pro}[defin]{Proposition}
\newtheorem{cor}[defin]{Corollary}
\theoremstyle{definition}

\newtheorem{rem}[defin]{Remark}

\newcommand{\N}{\mathbb{N}}
\newcommand{\Z}{\mathbb{Z}}
\newcommand{\R}{\mathbb{R}}

\renewcommand{\phi}{\varphi}
\newcommand{\sinc}{\text{sinc}}
\numberwithin{equation}{section}

\newcommand{\dsp}{\displaystyle}

\renewcommand\labelenumi{\emph{(\roman{enumi})}}
\renewcommand\theenumi\labelenumi

\titleformat{\section}
{\normalfont\fillast \fontsize{12}{15}\scshape}{\thesection.}{0.8em}{}

\titleformat{\subsection}
{\normalfont\fillast \fontsize{11}{12}\scshape}{\thesubsection.}{0.8em}{}

\begin{document}
	
\thispagestyle{plain}

\begin{center}
	\large
	{\uppercase{\bf Curves defined by a class of discrete operators: approximation result and applications}} \\ 
	\vspace*{0.5cm}
	{\scshape{Rosario Corso}$^{a,*}$, Gabriele Gucciardi$^b$}\\
	\vspace*{0.1cm}
	{\it $^a$Dipartimento di Matematica e Informatica, Università degli Studi di Palermo, 
		Palermo, Italy,}  rosario.corso02@unipa.it ($^*$corresponding author)\\
	\vspace*{0.1cm}
	{\it $^b$Università degli Studi di Palermo,  
		Palermo, Italy,} gabriele.gucciardi02@community.unipa.it
	
\end{center}

\normalsize 
\vspace*{1cm}	

\small 

\begin{minipage}{13.5cm}
	{\scshape Abstract.} In approximation theory classical discrete operators, like generalized sampling, Sz\'{a}sz-Mirak'jan, Baskakov and Bernstein operators, have been extensively studied for scalar functions. In this paper, we look at the approximation of curves by a class of discrete operators 
	and we exhibit graphical examples concerning several cases. The topic has useful implications about the computer graphics and the image processing: we discuss applications on the approximation and the reconstruction of curves in images.
\end{minipage}

\vspace*{.5cm}

\begin{minipage}{13.5cm}
\end{minipage}

\vspace*{.5cm}

\begin{minipage}{13.5cm}
\end{minipage}

\vspace*{0.5cm}
\normalsize

\section{Introduction}

Nowadays, several type of operators have been investigated to approximate bounded and continuous real valued functions of one variable. Here, for approximation we mean that a family of operators $\{S_n\}_{n\in \N}$, acting on some function space, has the property that 
\begin{equation}
	\label{eq_appr}
	\lim_{n\to +\infty}(S_n f)(t)=f(t), 
\end{equation}
when $f:I\to \R$ is a bounded continuous function defined on a interval $I\subset\R$ (not necessarily bounded) and $t\in I$. A typical framework consists of discrete operators
\begin{equation}
	\label{def_oper}
	(S_n f)(t)= \sum_{k\in J} f(\nu_{n,k}) K_{n,k}(t), \qquad f:I\to \R, \; t\in I
\end{equation}
defined in terms of the samples $f(\nu_{n,k})$ of $f$ in the point $\nu_{n,k}$ and the elementary functions $K_{n,k}$, where $n\in \N^+$ and $k$ varies in a finite or countable set $J$.
Classical examples include the generalized sampling operators \cite{BSS,RS,Sp} in the case $I=\R$, the Sz\'{a}sz-Mirak'jan \cite{Mirakjan,Szasz} and the Baskakov operators \cite{Baskakov} in the case $I=[0,+\infty[$ and the Bernstein operators \cite{L} for $I=[0,1]$. 

Since a vector-valued function $\gamma:I\to \R^d$, $d>1$, of one variable  is made up of $d$ scalar functions, i.e. the components, the operators above can be used also to approximate vector-valued functions, acting individually in each component.   
Using a common terminology, with {\it curve} we refer to a continuous vector-valued function on an interval. 

We mention that some of the previous listed operators have been considered for vector-valued functions in \cite{AG,Campiti,N,Tachev} and, moreover, Bernstein operators are at the base of the definition of the Bézier curves \cite{Farin}.

Taking into account the idea of approximation by components, in this paper we look at a class of discrete operators, their property of curves approximation and various supporting examples. Furthermore, we discuss some consequent applications in the context of computer graphics \cite{Farin} and image processing \cite{IP}. Nowadays, these fields play a crucial role in modern technology: computer graphics studies methods to generate and visualize images; image processing investigates operations for enhancing images or for extracting useful information from them. 
Even though a curve in an image is represented by a finite number of points, it can be thought as discretization of a curve in the real plane; hence, we can exploit the operators for its approximation. Anyway, other processes can be done for curves in images starting by the approximations, for instance the reconstruction with an increased resolution or affine transformations. We stress that performing these operations on the approximated curve (which possesses a mathematical expression defined by a continuous, and not discrete, variable) gives better results than those obtained applying the operations directly on the images.

Coming back to the theoric aspects, we state the convergence result for curves, considering a class of discrete operators $\{S_n\}_{n\in \N}$ in a general framework in which the generalized sampling, Sz\'{a}sz-Mirak'jan, Baskakov and Bernstein operators fit. In details, our setting about $\{S_n\}_{n\in \N}$ follows, with some modifications, the assumptions adopted in \cite{BM}. 
Let $I\subset \R$ be an interval (bounded or not), $n\in \N^+$, $J\subset \Z$ an at most countable index set and $\Gamma_n=(\nu_{n,k})_{k\in J}\subset I$ a sequence of points such that 
\begin{equation}
	\label{v_nk}
	\lambda_n <\nu_{n,k+1}-\nu_{n,k} \leq \Lambda_n \text{ for every }k\in J,
\end{equation}
where $\lambda_n ,\Lambda_n >0$ and $\dsp \lim_{n\to +\infty} \lambda_n=\lim_{n\to +\infty} \Lambda_n=0$. We consider a family of continuous functions $\{K_{n,k}\}_{n\in \N^+,k\in J}$, $K_{n,k}:I\to \R$, satisfying the following conditions
\begin{equation}
	\label{partition}
	\sum_{k\in J} K_{n,k}(t)=1 \text{ for every $n\in \N^+$ and $t\in I$},
\end{equation}
\begin{equation}
	\label{moment_uniform}
	 \sum_{k\in J} \left | K_{n,k}(t)\right |\text{ converges uniformly on compact sets of $I$ for every $n\in \N^+$},
\end{equation}
\begin{equation}
	\label{moment_bound}
\text{there exists $M_0$ such that }
	m_0(n):=\sup_{t\in I} \sum_{k\in J}\left | K_{n,k}(t)\right |  <M_0\text{ for every $n\in \N^+$},
\end{equation}
and 
\begin{equation}
	\label{limit_moment}
	 \lim_{n\to +\infty}\sum_{k\in J,|\nu_{n,k}-t|\geq \delta} \left | K_{n,k}(t)\right |=0 \text{ for every $\delta>0$ and $t\in I$}. 
\end{equation}
Under this setup, for every $n\in \N^+$ we define the operator $S_n$ as in \eqref{def_oper}. By the assumptions made, the convergence \eqref{eq_appr} for bounded and continuous real valued functions holds (Theorem \ref{thm_approx}) and it implies an analogous result (Corollary \ref{cor_appr_curves}), i.e. $\dsp \lim_{n\to +\infty} (S_n \gamma)(t) =\gamma(t)$,  for bounded curves $\gamma:I\to \R^d$, $\gamma(t)=(x_1(t),\dots,x_d(t))$, by defining $$(S_n \gamma)(t):=((S_n x_1)(t),\dots,(S_n x_d)(t)), \quad t\in I.$$

The paper is organized as follows. Section \ref{sec_thm} is devoted to the approximation results related to $\{S_n\}_{n\in \N}$ for real valued functions and for curves. In Section \ref{sec_spec} we show in more details that the generalized sampling, Sz\'{a}sz-Mirak'jan, Baskakov and Bernstein operators are special cases of our setting. Some explicit examples of approximations of curves in various cases (2D or 3D space, open or closed curves) are treated in Section \ref{sec_exm}. Finally, we discuss the applications about the approximation and the reconstruction of curves in images in Section \ref{sec_appl} providing methods and examples. 

\section{Approximation results}
\label{sec_thm}

In this section we state the main approximation results for the operators $S_n$, firstly for scalar functions and then for curves. For a bounded real function on a interval $I$ we write $\|f\|_\infty:=\sup_{t\in I}|f(t)|$, while for a bounded function $\gamma:I\to \R^d$, with components $x_1,\dots,x_d$, we write $\dsp \|\gamma\|_\infty:=\max_{i=1,\dots,d}(\|x_i\|_\infty)$. 
As preliminary note, we remark the following properties. 

\begin{pro} 
	\label{lem_approx}
	Let $f:I\to \R$ be a bounded function and $n\in \N^+$. The function $S_nf:I\to \R$ is well-defined, continuous and bounded.
\end{pro}
\begin{proof}
	For every $t\in I$, by the boundedness of $f$ and by \eqref{moment_bound} we have the following inequality 
	$$
	|(S_nf)(t)|\leq \sum_{k\in J} |f(v_{n,k})||K_{n,k}(t)|\leq \|f\|_\infty m_0(n)<\infty. 
	$$
	This proves that $S_nf$ is well-defined and bounded. \\
	Now, let $t\in I$ and $\epsilon>0$. Given $\rho>0$, by \eqref{moment_uniform} there exists $\overline{k}>0$ such that $\sum_{|k|>\overline{k}} |K_{n,k}(\tilde t)|<\epsilon$ for every $\tilde t\in I,|t-\tilde t|\leq \rho$. Moreover, by the continuity of $K_{n,k}$, for each $|k|\leq \overline{k}$ there exists $\delta_k>0$ such that $|K_{n,k}(t)-K_{n,k}(s)|<\frac{\epsilon}{2\overline{k}+1}$ for every $|t-s|<\delta_k$. Therefore, putting $\displaystyle \delta=\min(\{\rho\}\cup\{\delta_k:{|k|\leq \overline{k}}\})$, we have
	\begin{align*}
		|(S_nf)(t)-(S_nf)(s)|&\leq \left |\sum_{k\in J} f(\nu_{n,k}) (K_{n,k}(t)- K_{n,k}(s)) \right |\\
		&\leq \|f\|_\infty \sum_{k\in J} \left | K_{n,k}(t)- K_{n,k}(s)\right |\\
		&\leq \|f\|_\infty \left ( \sum_{|k|\leq \overline{k}} \left | K_{n,k}(t)- K_{n,k}(s)\right | +\sum_{|k|> \overline{k}} \left | K_{n,k}(t)- K_{n,k}(s)\right | \right )\\
		&\leq \|f\|_\infty (\epsilon +2\epsilon)=3\|f\|_\infty \epsilon
	\end{align*}
	for every $s\in I,|t-s|<\delta$. Thus, $S_nf$ is continuous. 
\end{proof}

The theorem below states that $S_n$ are approximation operators for scalar functions when $n$ goes to infinity (the proof follows standard steps, see for instance \cite[Theorem 1]{BM2}). We will next use it for formulating the result for curves. 

\begin{thm}
	\label{thm_approx}
	Let $f:I\to \R$ be a bounded function. Then $$\dsp \lim_{n\to +\infty} (S_n f)(t) = f(t),$$ for every point $t$ of continuity of $f$. Moreover, if $f$ is also uniformly continuous and \eqref{limit_moment} is satisfied uniformly with respect to $t\in I$, then
	$$\dsp \lim_{n\to +\infty} \|S_n f - f\|_\infty=0.$$
\end{thm}
\begin{proof}
	The steps to prove the two statements are very similar, so we confine to the proof of the second one. 
	Let $\epsilon>0$. Since $f$ is uniformly continuous there exists $\delta >0$ such that $|f(\nu_{n,k})-f(t)|<\epsilon$ for every $n\in \N^+,k\in J$ and $t\in I$ satisfying $|\nu_{n,k}-t|<\delta$. We denote by $J_1:=\{k\in J: |\nu_{n,k}-t|<\delta\}$, by $J_2:=\{k\in J: |\nu_{n,k}-t|\geq \delta\}$ and remark that $J_1$ is finite because of \eqref{v_nk}. 
	Moreover, by \eqref{partition}
	\begin{align*}
		(S_n f)(t) - f(t)=\sum_{k\in J} (f(\nu_{n,k})-f(t)) K_{n,k}(t),
	\end{align*}
	so we can write $|(S_n f)(t) - f(t)|\leq s_1(t)+s_2(t),$ where 
	$$
	s_1(t)=\sum_{k\in J_1} |f(\nu_{n,k})-f(t)| |K_{n,k}(t)|\;
	\text{ and }\;
	s_2(t)=\sum_{k\in J_2} |f(\nu_{n,k})-f(t)| |K_{n,k}(t)|.
	$$
	For hypothesis, \eqref{limit_moment} is satisfied uniformly with respect to $t\in I$. Hence, there exists $\overline{n}$ such that
	$\dsp \sum_{k\in J_2} \left | K_{n,k}(t)\right |<\epsilon$ for $n\geq \overline{n}$ and for any $t\in I$. 
	Thus, by \eqref{moment_bound} we have that for $n\geq \overline{n}$ and for any $t\in I$
	$$
	s_1(t)< \epsilon \sum_{k\in J_1} |K_{n,k}(t)|\leq  \epsilon\, m_0(n)\leq \epsilon\, M_0,
	$$
	and that
	$$
	s_2(t)< 2\|f\|_{\infty} \sum_{k\in J_2} |K_{n,k}(t)| < 2\|f\|_{\infty} \epsilon.
	$$
	Summarizing, $\|S_n f - f\|_{\infty}<(M_0+2\|f\|_\infty)\epsilon$ for $n\geq \overline{n}$ and so the proof is concluded. 
\end{proof}

In this paper, we are interested to curves in $\R^d$ and Proposition \ref{lem_approx} and Theorem \ref{thm_approx} give, as an immediate consequence, the following approximation result. For a bounded function $\gamma: I\to \R^d$, $\gamma(t)=(x_1(t),\dots,x_d(t))$, $t\in I$, we define $S_n \gamma : I\to \R^d$ by 
\begin{equation*}	
	\label{S_n_curve}
	(S_n \gamma)(t):=((S_n x_1)(t),\dots,(S_n x_d)(t)), \qquad t\in I,
\end{equation*}
or, in a more compact form, by
\begin{equation*}
	(S_n \gamma)(t)= \sum_{k\in J} \gamma(\nu_{n,k}) K_{n,k}(t), \qquad t\in I.
\end{equation*}

\begin{cor}\label{cor_appr_curves}
	Let $\gamma: I\to \R^d$ be a bounded function. The following statements hold.
	\begin{enumerate}
		\item $S_n \gamma$ is a bounded and continuous function (i.e., a bounded curve);
		\item $\dsp \lim_{n\to +\infty} (S_n \gamma)(t) =\gamma(t)$ for every point $t$ of continuity of $\gamma$;
		\item if $\gamma$ is also uniformly continuous and \eqref{limit_moment} is satisfied uniformly with respect to $t\in I$, then $\dsp \lim_{n\to +\infty} \|S_n \gamma - \gamma\|_\infty=0$.
	\end{enumerate}
\end{cor}

\section{Special cases of operators}
\label{sec_spec}

As mentioned in the introduction, special cases of our operators setting are the generalized sampling operators, the  Sz\'{a}sz-Mirak'jan operator, the Baskakov operator and the Bernstein operator.  We give the details for each of them below.

\subsection{Generalized sampling operator}

We start recalling that, following \cite{RS}, a continuous function $\chi:\R\to \R$ is called a {\it kernel} if the following conditions are satisfied:
\begin{enumerate}
	\item[(i)] for every $t\in \R$, 
	\begin{equation*}
		\sum_{k\in \Z}\chi(t-k)=1;
	\end{equation*}
	\item[(ii)] 
\begin{equation*}
	\sum_{k\in \Z} |\chi(t-k)|   \text{ converges uniformly for }t\in [0,1]. 
\end{equation*}
\end{enumerate}

With equation \eqref{def_oper} and the choice $I=\R$, $J=\Z$, $\nu_{n,k}=\frac kn$ and  $K_{n,k}(t)=\chi(nt-k)$, where $\chi$ is a kernel, we recover the {\it generalized sampling operator}
$$
(\mathcal{S}^\chi_{n} f)(t) = \sum_{k\in \Z} f\! \left (\frac kn \right)\chi(nt-k), \qquad f:\R\to \R,\; t\in \R. 
$$
This operator has been studied in \cite{Sp,RS,BSS} and under some variations in \cite{ACLR,ACCSV,BBSV,BCV,CCCGV,CCNV,CG,CDV,CPV,TVSS,VZ}.
Assumption \eqref{partition} follows by property (i), assumptions \eqref{moment_uniform} and \eqref{moment_bound} are given by Lemma 1(b) of \cite{RS}, finally \eqref{limit_moment} holds (also uniformly with respect to $t\in \R$) by Lemma 1(c) of \cite{RS}. 

Classical examples of kernels are the Fejér kernel and the B-splines \cite{RS,BBSV}. 
The {\it Fejér kernel}, shown in Figure \ref{Fejer}, is $$F(t)=\frac 1 2 \sinc^2 \!\left (\frac t2 \right ),$$
where $\sinc$ is the function defined by
$\dsp 
\sinc(t)=\begin{cases}
	\frac{\sin(\pi t)}{\pi t }  & t\neq 0 \\
	1 & t=0.
\end{cases} 
$

\noindent 
The {\it B-spline of order $m$} is the function defined by $$M_1(t)=\begin{cases}
	1 & t\in \left [-\frac 12, \frac 12\right ] \\
	0 & \text{otherwise},
\end{cases} $$
and
$$M_m(t)=\frac{1}{(m-1)!} \sum_{j=0}^m (-1)^j \binom{m}{j} \max \left(\frac m2 +t-j,0\right )^{m-1}, \qquad m\geq 2,$$
or, equivalently, by the recursive formula
$$
M_m=M_{m-1}\ast M_1, \qquad m\geq 2,
$$
where $\ast$ is the convolution product. In contrast to the Fejér kernel, the B-spline of order $m$ has compact support (namely $[-\frac m2 ,\frac m2]$). 
The graphs of some B-splines are displayed in Figure \ref{Bsplines}.

\begin{figure}[h]
	\begin{center}
		\begin{subfigure}[b]{0.45\textwidth}
			\includegraphics[scale=0.45]{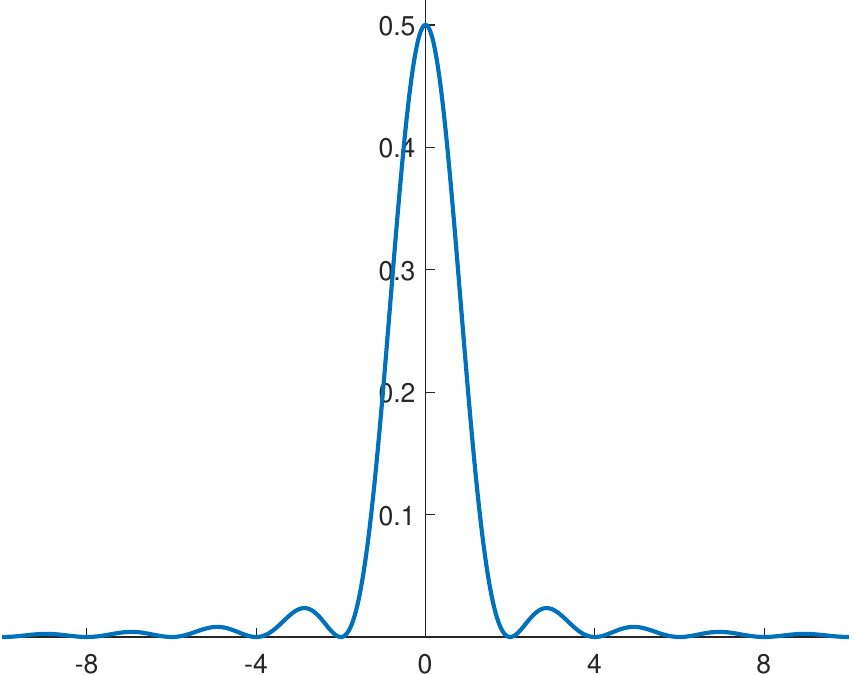}
			\caption{Fejér kernel.\\{\color{white}.}}  
			\label{Fejer}
		\end{subfigure}
		\hspace*{1cm}
		\begin{subfigure}[b]{0.45\textwidth}
			\includegraphics[scale=0.45]{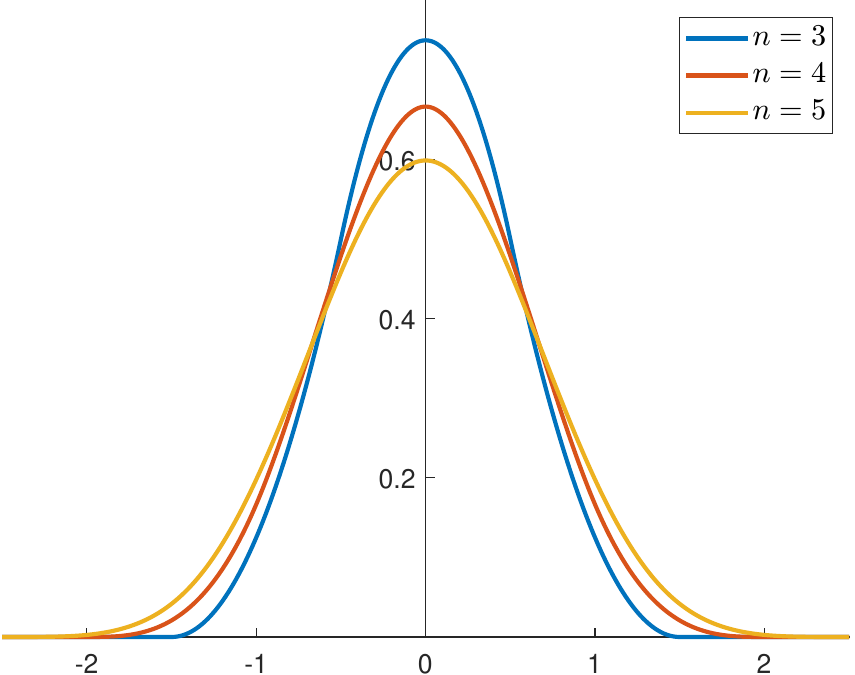}
			\caption{B-splines of order $3$, $4$ and $5$ (in blue, red and orange, respectively).}  
			\label{Bsplines}
		\end{subfigure}
	\end{center}
\end{figure}

\subsection{Sz\'{a}sz-Mirak'jan operator}

From \eqref{def_oper} and with $I=[0,+\infty[$, $J=\N$, $\nu_{n,k}=\frac kn$ and  $K_{n,k}(t)=e^{-nt}\frac{(nt)^k}{k!}$ we get the
{\it Sz\'{a}sz-Mirak'jan operator}
$$
(\mathsf{S}_{n} f)(t) = \sum_{k=0}^{+\infty} f\! \left (\frac kn \right)e^{-nt}\frac{(nt)^k}{k!}, \qquad f:[0,+\infty[\;\to \R, \; t\in [0,+\infty[.
$$

Sz\'{a}sz-Mirak'jan operator has been studied in \cite{Mirakjan,Szasz} and, with some modifications, also in \cite{Acar,ACL,AGSK}. 

The continuous functions $K_{n,k}$ satisfy the assumptions. Indeed, \eqref{partition} follows by the power series of the exponential function and trivially implies  \eqref{moment_bound} since $K_{n,k}>0$ for every $n\in \N^+$ and $k\in \N$. 
Moreover, let $n\in \N^+$, $U$ a compact of $[0,+\infty[$ and $\epsilon>0$. There exists $0\leq a < b$ such that $U\subset [a,b]$ and, since $K_{n,k}$ is strictly increasing on $[0,\frac kn]$, 
there exists $\tilde k_1$ such that 
$$\sup_{t\in U} K_{n,k}(t) \leq \sup_{t\in [a,b]} K_{n,k}(t)=K_{n,k}(b)=e^{-nb}\frac{(nb)^k}{k!}, \qquad \text{for all }k>\tilde k_1.$$
Furthermore, there exists $\tilde k_2$ such that  $\dsp \sum_{k>\tilde k_2} e^{-nb}\frac{(nb)^k}{k!}<\epsilon$. 
Therefore, if $\tilde k=\max(\tilde k_1,\tilde k_2)$, 
$$
\sum_{k>\tilde k} K_{n,k}(t)\leq \sum_{k>\tilde k} e^{-nb}\frac{(nb)^k}{k!}<\epsilon, \quad \text{for all }t\in U.
$$
In conclusion, also \eqref{moment_uniform} holds true. Finally, \eqref{limit_moment} is satisfied because the proof of \cite[Corollary 2]{BM} shows that for $\delta >0$
$$\dsp \sum_{k\in \N,|\nu_{n,k}-t|\geq \delta}  K_{n,k}(t) =o(n^{-1}) \text{ for } n\to +\infty.$$

\subsection{Baskakov operator}

The choice $I=[0,+\infty[$, $J=\N$, $\nu_{n,k}=\frac kn$ and  $K_{n,k}(t)=\binom{n+k-1}{k}\frac{t^k}{(1+t)^{n+k}}$ leads to the {\it Baskakov operator} (introduced in \cite{Baskakov}, see also \cite{AMM,AGSK,GAS})
$$
(\mathsf{B}_{n} f)(t) = \sum_{k=0}^{+\infty} f \! \left (\frac kn \right)\!\binom{n+k-1}{k}\frac{t^k}{(1+t)^{n+k}}, \qquad f:[0,+\infty[\;\to \R,\; t\in [0,+\infty[.
$$

Concerning the assumptions about the continuous functions $K_{n,k}$, \eqref{partition} holds as consequence of the identity $\dsp \frac{1}{1-q}=\sum_{k=0}^{+\infty} q^k$, with $q=\frac{t}{1+t}$. Then, \eqref{moment_bound} is verified since $K_{n,k}>0$ for every $n\in \N^+$ and $k\in \N$. To prove \eqref{moment_uniform} one can follows steps similar to those made for Sz\'{a}sz-Mirak'jan operators. Finally, in \cite[Section 5.2]{BM} it is proved that for every $\delta>0$ and $t\in I$ 
$$\sum_{k\in J,|\nu_{n,k}-t|\geq \delta}  K_{n,k}(t) (\nu_{n,k}-t)^2=o(n^{-2}) \quad\text{as $n\to +\infty$.}$$ 
Thus, since $\delta^2 K_{n,k}(t) \leq K_{n,k}(t) (\nu_{n,k}-t)^2$ if $|\nu_{n,k}-t|\geq \delta$, the remaining assumption \eqref{limit_moment} holds too.

\subsection{Bernstein operator}

If $I=[0,1]$, $J=\{0,1,\dots,n\}$, $\nu_{n,k}=\frac kn$ and  $K_{n,k}(t)=\binom{n}{k}t^k(1-t)^{n-k}$ we recover the {\it Bernstein operator} \cite{L} (see also \cite{AMCY,BMMA,CW,MAK} for related works)
$$
(\mathcal{B}_{n} f)(t) = \sum_{k=0}^n f \! \left (\frac kn \right)\!\binom{n}{k}t^k(1-t)^{n-k}, \qquad f:[0,1]\to \R,\; t\in [0,1].
$$
The assumption \eqref{partition} can be easily checked and  \eqref{moment_uniform}-\eqref{limit_moment} are trivially satisfied since $J$ is finite and $I$ is bounded. 
We also remark that if $\gamma:[0,1]\to \R^d$ is a curve, then $\mathcal{B}_{n} \gamma$ is a so-called Bézier curve \cite{Farin} with control points $\{ \gamma \left (\frac kn \right)\}_{k\in J}$. \\

Figure \ref{fig_elem_funct} shows the graphs of some elementary functions appearing in the formulations of the generalized sampling, Sz\'{a}sz-Mirak'jan, Baskakov and Bernstein operators.

\begin{figure}[h!]
	\begin{center}
		\begin{subfigure}[b]{0.45\textwidth}
			\includegraphics[scale=0.45]{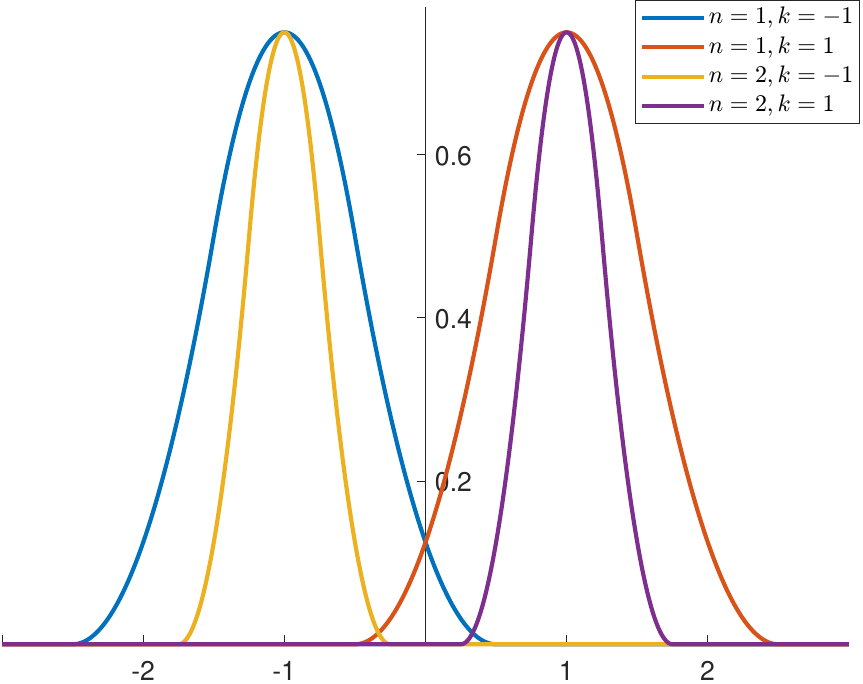}
			\caption{Some functions $t\mapsto M_3(nt-k)$ for different values of $n$ and $k$.\medskip }
		\end{subfigure}
		\hspace*{.5cm}
		\begin{subfigure}[b]{0.45\textwidth}
			\includegraphics[scale=0.45]{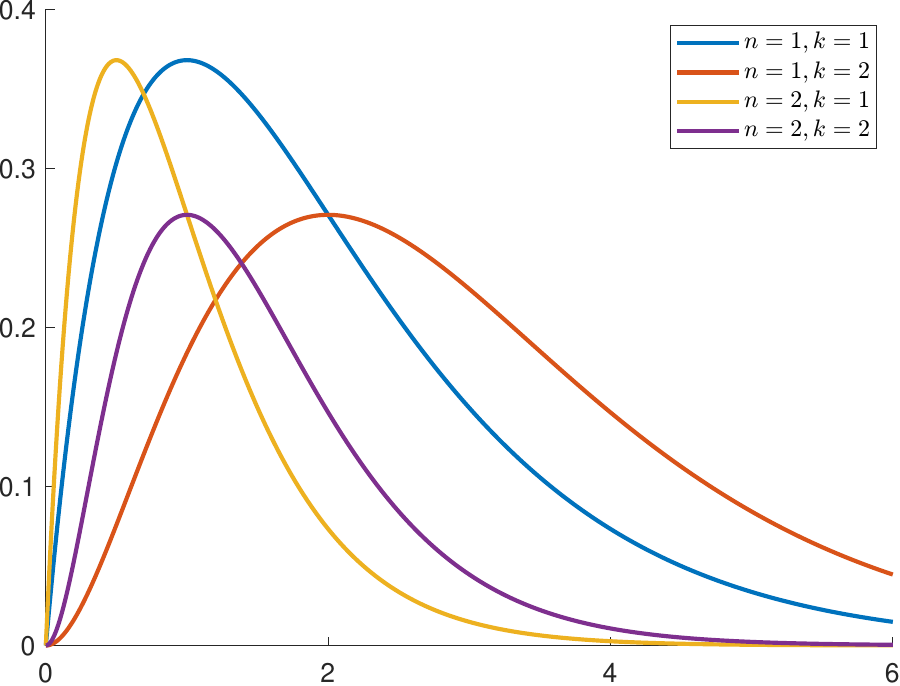}
			\caption{Some functions $t\mapsto e^{-nt}\frac{(nt)^k}{k!}$ for different values of $n$ and $k$. \medskip}
		\end{subfigure}
		\begin{subfigure}[b]{0.45\textwidth}
			\includegraphics[scale=0.45]{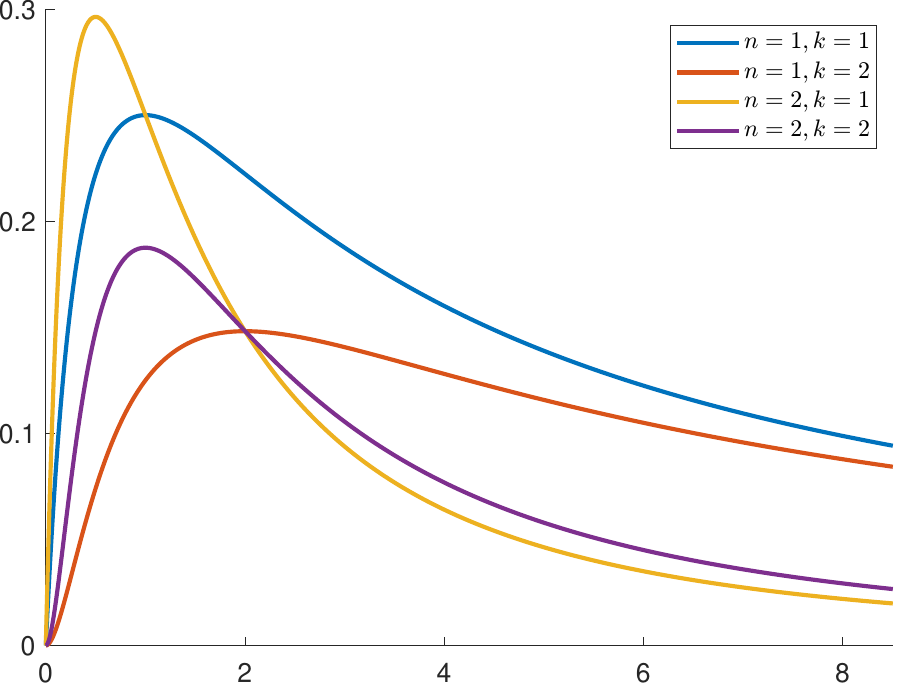}
			\caption{Some functions $t\mapsto \binom{n+k-1}{k}\frac{t^k}{(1+t)^{n+k}}$ for different values of $n$ and $k$. \smallskip}
		\end{subfigure}
		\hspace*{.5cm}
		\begin{subfigure}[b]{0.45\textwidth}
			\includegraphics[scale=0.45]{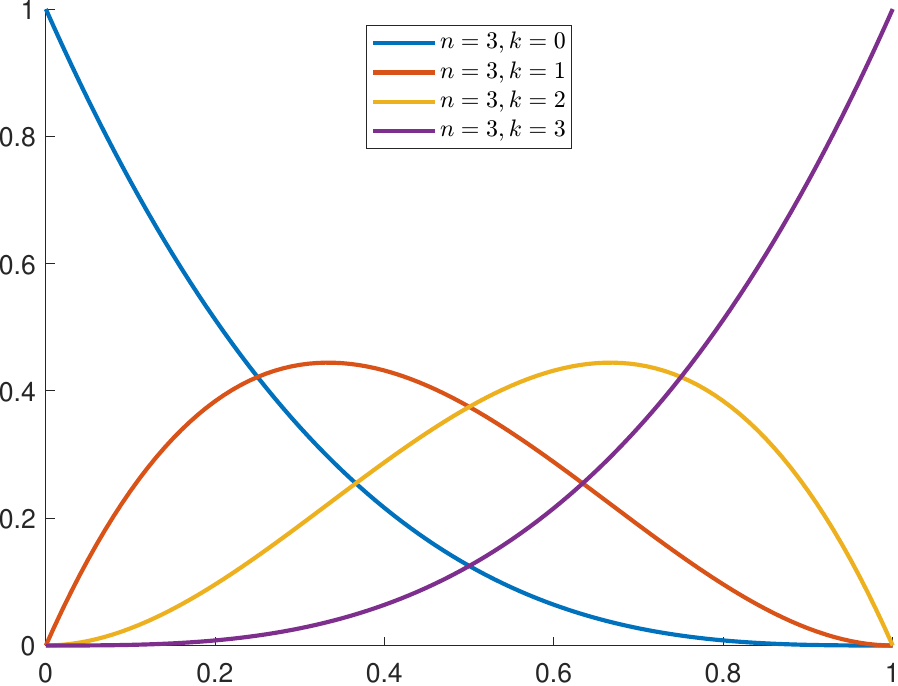}
			\caption{Some functions $t\mapsto \binom{n}{k}t^k(1-t)^{n-k}$ for different values of $n$ and $k$.\smallskip }
		\end{subfigure}
	\end{center}
	\caption{Each sub-figure shows some elementary functions $K_{n,k}$ appearing in the expression of (a) generalized sampling operator with B-spline kernel of order $3$; (b) Sz\'{a}sz-Mirak'jan operator; (c) Baskakov operator; (d) Bernstein operator.}
	\label{fig_elem_funct}
\end{figure}

\section{Examples}
\label{sec_exm}

In this section we illustrate some graphical examples of the approximation of curves, as stated by Corollary \ref{cor_appr_curves}, via the operators of type \eqref{def_oper}. We first remark that a curve with bounded domain can be approximated by the operator $S_n$ defined by \eqref{def_oper} even if the domain of the curve is not the interval $I$. To explain better this idea with some examples, let us assume that $\gamma:[a,b]\to \R^d$ is a bounded curve where $[a,b]$ is a bounded interval ($a<b$). We have to analyze different cases depending on the type of interval $I$.  \\

If $I$ is a bounded closed interval, then we can consider the new curve $\tilde \gamma:I\to \R^d$, $\tilde \gamma=\gamma \circ \sigma$, where $\sigma:I\to [a,b]$ is a bijective function (for instance, an affine transformation between the domains). Since $\tilde \gamma$ has domain $I$, we can approximate it by a curve $S_n\tilde \gamma $. Finally, to come back to the domain $[a,b]$, we take into account the inverse transformation $\sigma^{-1}$. In other words, we consider the curves $S_n\tilde \gamma\circ \sigma^{-1}=S_n (\gamma \circ \sigma)\circ \sigma^{-1}$ in order to approximate $\gamma$. The result follows by the fact that for every $t\in [a,b]$
$$\dsp \lim_{n\to +\infty} (S_n (\gamma \circ \sigma)\circ \sigma^{-1})(t)=\lim_{n\to +\infty} (S_n (\gamma \circ \sigma))(\sigma^{-1}(t))=(\gamma \circ \sigma)(\sigma^{-1}(t))=\gamma(t).$$

This is exactly the case of Bernstein operators; in fact, $\mathcal{B}_{n}$ is defined for functions with domain $I=[0,1]$, so to approximate a curve $\gamma:[a,b]\to \R^d$ we can take the curve $\mathcal{B}_n (\gamma \circ \sigma)\circ \sigma^{-1}$ where $\sigma:I \to [a,b]$ is given by $\sigma(s)=(1-s)a+sb$.

If, instead, $I=\R$ (in particular, in the case of generalized sampling operators) we do not have necessity of transform $[a,b]$. Indeed, since $[a,b]$ is a subset of $\R$, we can simply extend the curve $\gamma$ to a new curve $\tilde \gamma :\R \to \R^d$ setting 
\begin{equation}\label{estensione}
 \tilde \gamma(t)=\begin{cases}
	\gamma(a) &\quad t<a\\
	\gamma(t) &\quad t\in [a,b]\\
	\gamma(b) &\quad t>b.
\end{cases}
\end{equation}
Thus, to have an approximation of $\gamma$ we can simply take an approximation $S_n\tilde \gamma$ of $\tilde \gamma$ and then a restriction of it in the interval $[a,b]$. Note that we have extended $\gamma$ as in \eqref{estensione} to have a continuous function (so a curve) and then to apply Corollary \ref{cor_appr_curves} for every point of $\R$. \\

Finally, let us consider $I=[0,+\infty[$ as in the case of Sz\'{a}sz-Mirak'jan or Baskakov operators. First of all, $[a,b]$ is not necessarily a subset of $I$, so we can appeal to the translation $\sigma:[0,b-a]\to [a,b]$, $\sigma(s)=s+a$, and  secondly, we need an extension to the interval $[0,+\infty[$. Summarizing, we define the curve
\begin{equation}\label{estensione_dx}
	\tilde \gamma(t)=\begin{cases}
		\gamma \circ \sigma(t)=\gamma(t+a) &\quad t\in [0,b-a]\\
		\gamma(b) &\quad t>b-a,
	\end{cases}
\end{equation}
we approximate it with $S_n \tilde \gamma$, we take the restriction $R{(S_n \tilde \gamma)}$ of $S_n \tilde \gamma$ on the interval $[0,b-a]$ and, coming back to the original domain exploiting $\sigma^{-1}$, we are able to  approximate $\gamma$ with the curve $R{(S_n \tilde \gamma)} \circ \sigma^{-1}$.

\begin{figure}[h]
	\begin{center}
		\begin{subfigure}[b]{0.42\textwidth}
			\includegraphics[scale=0.52]{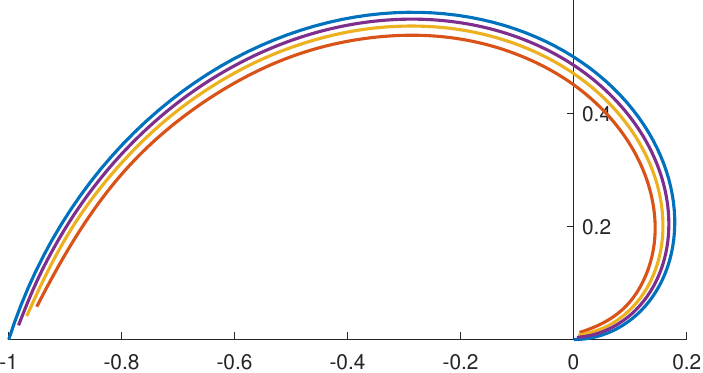}
			\caption{Approximations with the generalized sampling operator and Fejér kernel (red: $n=15$, orange: $n=25$, purple: $n=50$).}
			\label{}
		\end{subfigure}
	\hspace*{0.5cm}
		\begin{subfigure}[b]{0.42\textwidth}
			\includegraphics[scale=0.52]{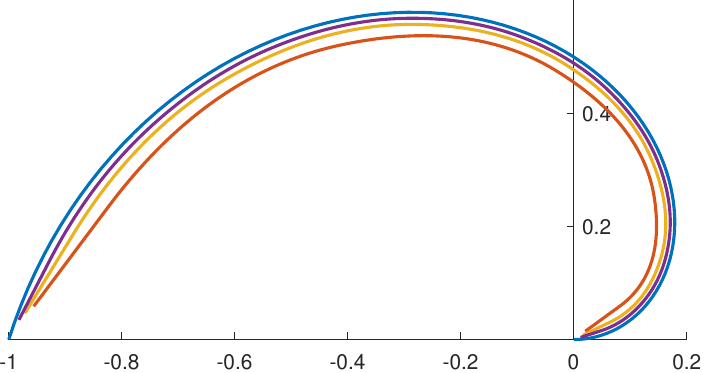}
			\caption{Approximations with the generalized sampling operator and B-spline kernel of order $3$ (red: $n=5$, orange: $n=7$, purple: $n=10$).}  
			\label{}
		\end{subfigure}
	\label{}
\end{center}
	\begin{center}
		\begin{subfigure}[b]{0.42\textwidth}
			\includegraphics[scale=0.52]{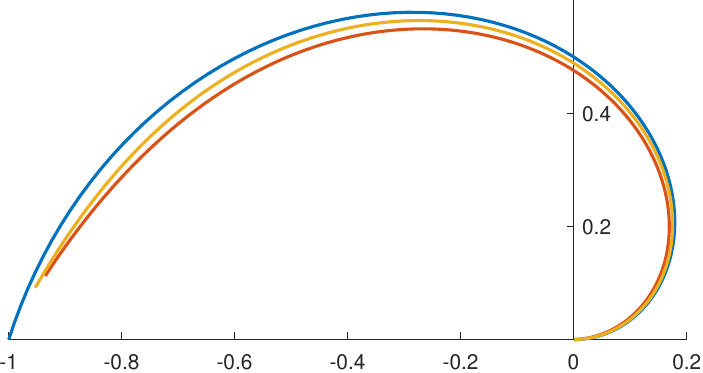}
			\caption{Red: approximation with the Sz\'{a}sz-Mirak'jan operator ($n=90$). Orange: approximation with the Baskakov operator ($n=300$).}
			\label{}
		\end{subfigure}
		\hspace*{0.5cm}
		\begin{subfigure}[b]{0.42\textwidth}
			\includegraphics[scale=0.52]{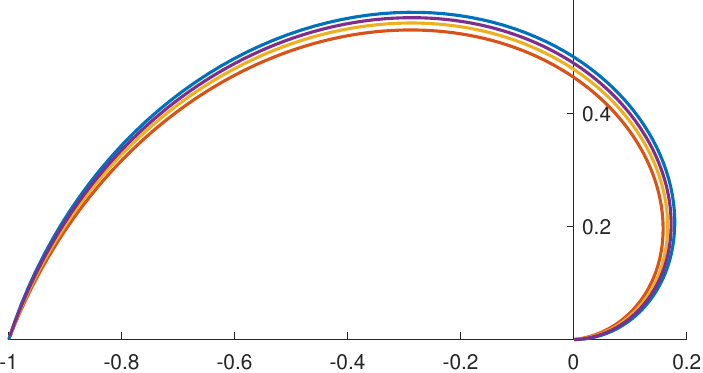}
			\caption{Approximations with the Bernstein operator (red: $n=30$, orange: $n=50$, purple: $n=100$).\\}  
			\label{}
		\end{subfigure}
	\caption{In each sub-figures the open curve in blue is $\gamma:[0,1]\to \R^2$, $\gamma(t)=(t\cos(\pi t),t\sin(\pi t))$. The curves in red, orange and purple are some approximations given by different operators.}
		\label{appr_open}
	\end{center}
\end{figure}

As example, in Figure \ref{appr_open} we consider a curve of the real plane (in blue) and its approximations by the operators of Section \ref{sec_spec} for different values of the parameter $n$. In particular, we apply the method explained above in the cases of generalized sampling, Sz\'{a}sz-Mirak'jan and Baskakov operators. In Figure \ref{appr_3d} we show a curve in the tridimensional space and an approximation by the generalized sampling operator (the kernel chosen is the B-spline of order $3$ and $n$ is equal to $10$). \\

\begin{figure}[h]
	\begin{center}
		\includegraphics[scale=0.6]{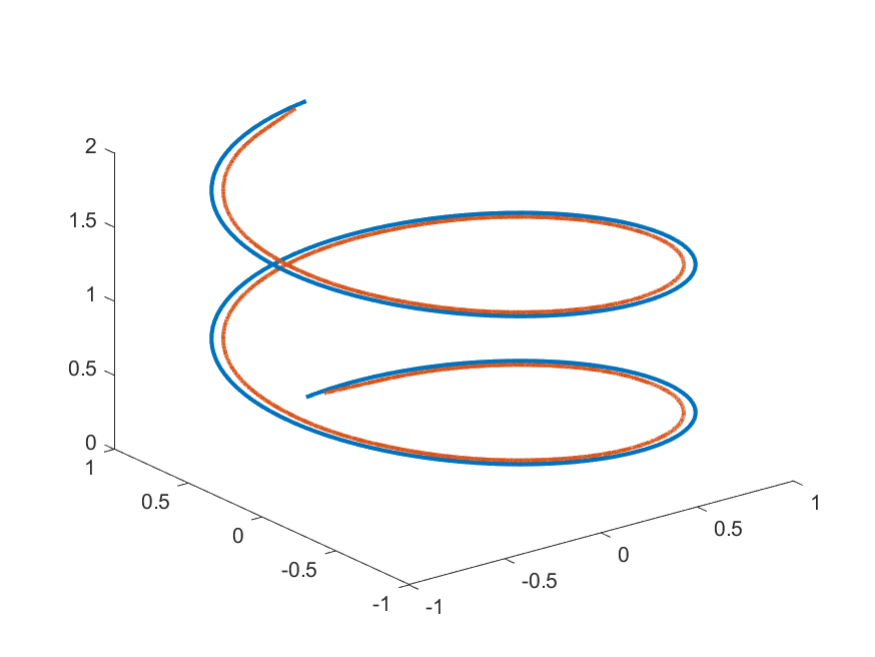}
		\caption{The curve in blue is $\gamma:[0,2]\to \R^3$, $\gamma(t)=(\cos(2\pi t),\sin(2\pi t),t)$. The curve in red is the approximation given by the generalized sampling operator with B-spline kernel of order $3$ ($n=10$).}
		\label{appr_3d}
	\end{center}
\end{figure}

If $\gamma:[0,1]\to \R^d$ is, in particular, a closed curve (i.e. $\gamma(0)=\gamma(1)$ is satisfied), we would like approximations of $\gamma$ which are closed too. For this end, we extend $\gamma$ not as in \eqref{estensione}, but periodically, i.e. we define
\begin{equation}\label{estensione_period}
	\tilde \gamma(t)=\gamma(t-m),\, \text{ for every }t\in \R,m\in \Z \text{ such that }t-m\in [0,1].
\end{equation}
Hence, the approximation $\mathcal{S}^\chi_{n} \tilde \gamma$ with the generalized sampling operator is a closed curve as a consequence of the following result.

\begin{pro}
	\label{pro_period}
	Let $f:\R\to \R$ be a bounded and periodic function with period $1$. Then, for every kernel $\chi$ and every $n\in \N^+$, the function $\mathcal{S}^\chi_{n}f$ is periodic with period $1$.
\end{pro}
\begin{proof}
	Let $\chi$ be a kernel and $n\in \N^+$. For every $t\in \R$ and $m\in \Z$ we have
	\begin{align*}	
		(\mathcal{S}^\chi_{n} f)(t-m) &= \sum_{k\in \Z} f\! \left (\frac kn \right)\chi(n(t-m)-k)\\
		&= \sum_{k\in \Z} f\! \left (\frac kn \right)\chi(nt-nm-k)\\
		&= \sum_{k\in \Z} f\! \left (\frac{k+nm}{n} \right)\chi(nt-nm-k)
	\end{align*}
where in the last equality we used the periodicity of $f$. Thus, with a change $k'=k+nm$, we can write
$$
	(\mathcal{S}^\chi_{n} f)(t-m) = \sum_{k'\in \Z} f\! \left (\frac{k'}{n} \right)\chi(nt-k')= (\mathcal{S}^\chi_{n} f)(t),
$$
i.e. $\mathcal{S}^\chi_{n}f$ is periodic with period $1$.
\end{proof}

\newpage
We remark that if $\gamma:[0,1]\to \R^d$ is a closed curve, then also the approximations $\mathcal{B}_{n} \gamma$ by the Bernstein operator are closed curve, because $(\mathcal{B}_{n}\gamma)(0)=\gamma(0)=\gamma(1)=(\mathcal{B}_{n}\gamma)(1)$. 
In Figure \ref{appr_closed} a closed curve of the real plane and some approximations by the generalized sampling operator (making use of Proposition \ref{pro_period}) and the Bernstein operator are presented. 

\begin{figure}[h!]
	\begin{center}
		\begin{subfigure}[b]{0.44\textwidth}
			\begin{center}
				\includegraphics[scale=0.55]{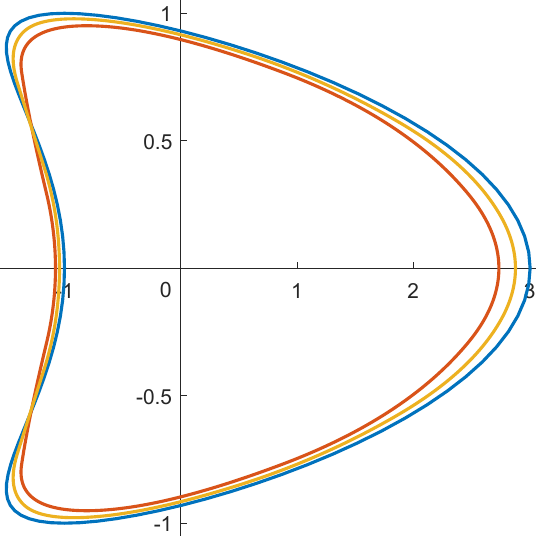}
			\end{center}
			\caption{Approximations with the generalized sampling operator and B-spline kernel of order $3$ (red: $n=10$, orange: $n=15$).
			}  
			\label{}
		\end{subfigure}
		\label{}
		\hspace*{0.2cm}
		\begin{subfigure}[b]{0.44\textwidth}
			\begin{center}
				\includegraphics[scale=0.55]{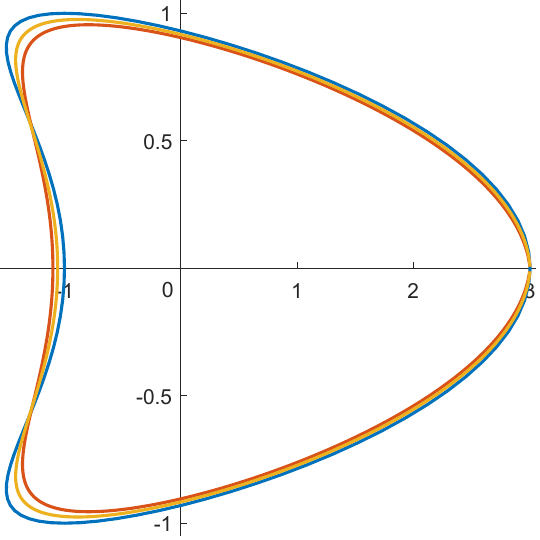}
			\end{center}
			\caption{Approximations with the Bernstein operator (red: $n=80$, orange: $n=150$).
				\\}  
			\label{}
		\end{subfigure}
	\end{center}
	\caption{In each sub-figures the closed curve in blue is $\gamma:[0,1]\to \R^2$ defined by $\gamma(t)=(\cos(4\pi t)+2\cos(2\pi t),\sin(2\pi t))$. The curves in red, orange and purple are some approximations given by different operators.}
	\label{appr_closed}
\end{figure}

\section{Applications}
\label{sec_appl}

The approximation of curves by discrete operators \eqref{def_oper} lends itself to some applications about computer graphics and image processing. In these fields several types of curves have found utility for problems of approximations, for instance trigonometric polynomials \cite{KG}, Bezier curves, splines and rational Bezier curves \cite{Farin}, NURBS curves \cite{NURBS}. Coming back to the general operators \eqref{def_oper}, we want to discuss some direct consequences of Corollary \ref{cor_appr_curves}. 

To begin with, we recall that a gray-scale image is, under the mathematical point of view, a matrix containing in each entry ({\it pixel}) the corresponding level gray. Curves in images (like, for instance, object contours or sharp transitions of gray levels) can be considered as discretization\footnote{For discretization we mean the process that round off the coordinates values towards the nearest integer.} of curves in the real plane. In particular, if we are interested only to a curve itself, then we can represents it in a {\it binary} image, i.e. a matrix with values $0$ (for the background) and $1$ (for the curve). In this paper, we consider only image curves which are discretization of curves which does not intersect themselves. Figure \ref{contorno_image_base} shows a binary image with value $0$ in the entries represented by a white square and value $1$ in the entries represented by a colored square. The pixels with value equals to $1$ constitute a closed curve. 

\begin{figure}[h]
	\begin{center}
		\includegraphics[scale=0.4]{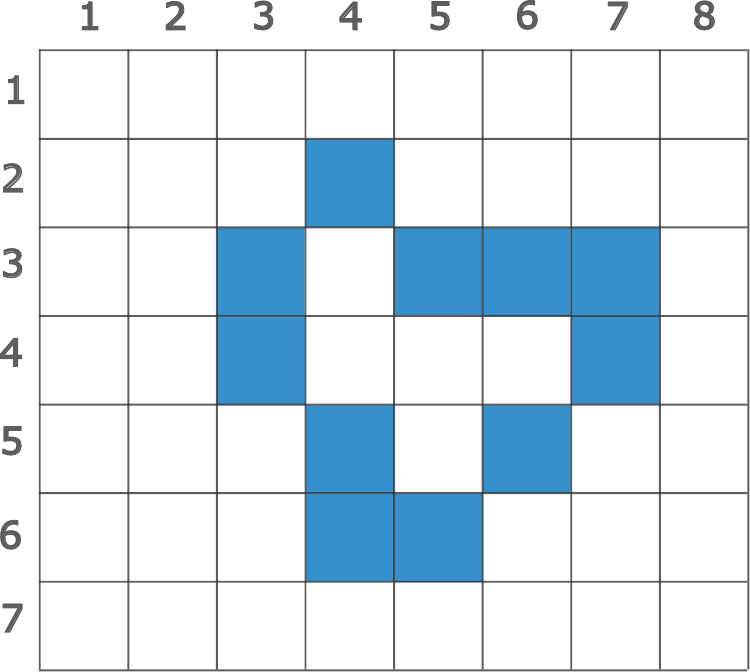}
		\caption{An example of a curve in an image.} 
	\label{contorno_image_base}
\end{center}
\end{figure}

\vspace*{-0.4cm}

\subsection{Approximations of curves in images}

The approximation method of curves in images by means of an operator $S_n$ of type \eqref{def_oper} consists in the extraction of the coordinates of the curve points, from which two continuous functions $x_1:I\to \R$ and $x_2:I\to \R$ can be defined, and then in the approximation of $x_1$ and $x_2$ in terms of $S_n$. By Corollary \ref{cor_appr_curves}, the curve $S_n\gamma=(S_n x_1,S_n x_2)$ is, for $n$ large enough, an approximation of the curve $\gamma=(x_1,x_2)$ and, as consequence, it determines an approximation of the image curve via discretization.  

With the help of Figure \ref{contorno_image} we give more details, confining to the case of closed curves (the other case, open curves, can be treated with little and intuitive changes). 
First of all, Figure \ref{directions} shows a pixel, the eight directions to the nearest neighborhoods and, following \cite{Freeman} (see also \cite{KG}), the ordering from $0$ to $7$ of them. Figure \ref{ex_percorso} contains the same curve of Figure \ref{contorno_image_base} (the square colored in red is the starting and final point of the curve which is, by our convention, the point of the contour with maximum ordinate and minimum abscissa). Based on the ordering of Figure \ref{directions}, a path that runs through the curve from and to the initial point can be created as in Figure \ref{ex_percorso} (among the possible directions, the chosen one is represented by the lowest number).
 
 \begin{figure}[h]
 	\begin{center}
 		\hspace*{0.2cm}
 		\begin{subfigure}[b]{0.3\textwidth}
 			\includegraphics[scale=0.55]{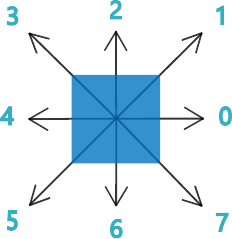} \vspace*{1.4cm}
 			\caption{\hspace*{2.1cm}}  
 			\label{directions}
 		\end{subfigure}
 		\begin{subfigure}[b]{0.4\textwidth}
 			\hspace*{0.1cm}
 			\includegraphics[scale=0.4]{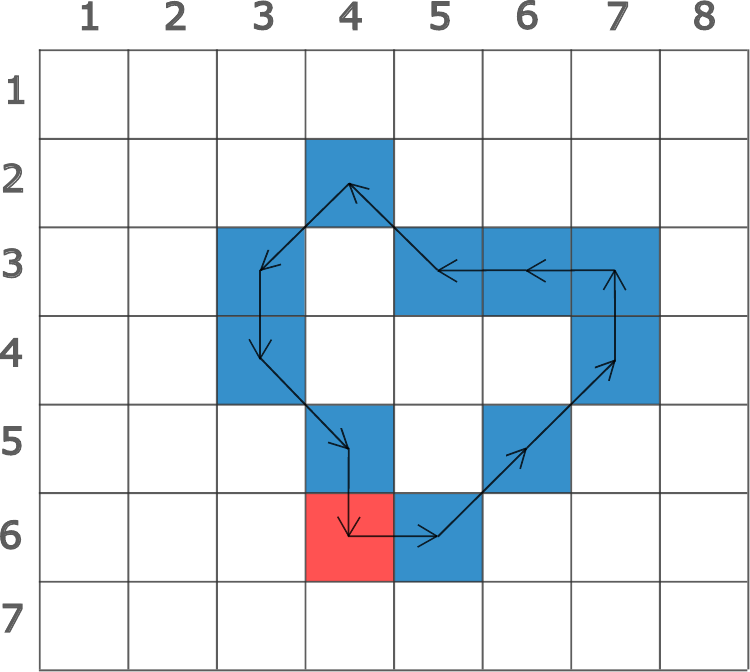}
 			\vspace*{0.05cm}
 			\caption{}
 			\label{ex_percorso}
 		\end{subfigure}
 		\caption{(a) The eight directions from a pixel to the nearest neighborhoods and the corresponding ordering as defined in \cite{Freeman}. (b) An image representing a closed curve with starting and final point in red. The path indicated by the arrows in the second figure is obtained starting from the red pixel and moving to a nearest neighborhood following the ordering of the directions.} 
 	\label{contorno_image}
 \end{center}
\end{figure}

Collecting the labels of the directions as a sequence we obtain the so-called {\it chain code} \cite{KG}. Anyway, for our applications, we are actually interested to the sequences of abscissas and ordinates\footnote{For a point of coordinates $(i,j)$ of an image curve we call $j$ and $i$ the abscissa and the ordinate of the point, respectively. Note also that we enumerate the ordinates from above to bottom (see Figure \ref{ex_percorso}) as it usually done for matrices, i.e. for the representation of images.}. For example, the curve represented in  Figure \ref{ex_percorso}, and with starting point in red, has chain code $c=(0,1,1,2,4,4,3,5,6,7,6)$, sequence of abscissas $u=(4,5,6,7,7,6,5,4,3,3,4)$ and sequence of ordinates $v=(6,6,5,4,3,3,3,2,3,4,5)$.

Once the sequences of abscissas $u=(u_j)_{j=1}^N$ and of ordinates $v=(v_j)_{j=1}^N$ are determined, we define a piece-wise linear function $x_1:[0,1]\to \R$ associated to $u$ as 
\begin{equation}
	\label{coord_function}
\begin{cases}
x_1(t)=(j-Nt)(u_j-u_{j+1})+u_{j+1} \quad&\text{if } t\in \left[\frac{j-1}{N},\frac{j}{N} \right[\text{ and } 1\leq j\leq N-1,\\
x_1(t)=N(1-t)(u_N-u_{1})+u_{1} \quad&\text{if } t\in\left [\frac{N-1}{N},1 \right]
\end{cases}
\end{equation}
and, similarly, a function $x_2:[0,1]\to \R$ associated to $v$. In Figure \ref{abs_ord_percorso} we show the functions $x_1$ and $x_2$ for the example of Figure \ref{ex_percorso}.

\begin{figure}[h!]
	\begin{center}
		\includegraphics[scale=0.45]{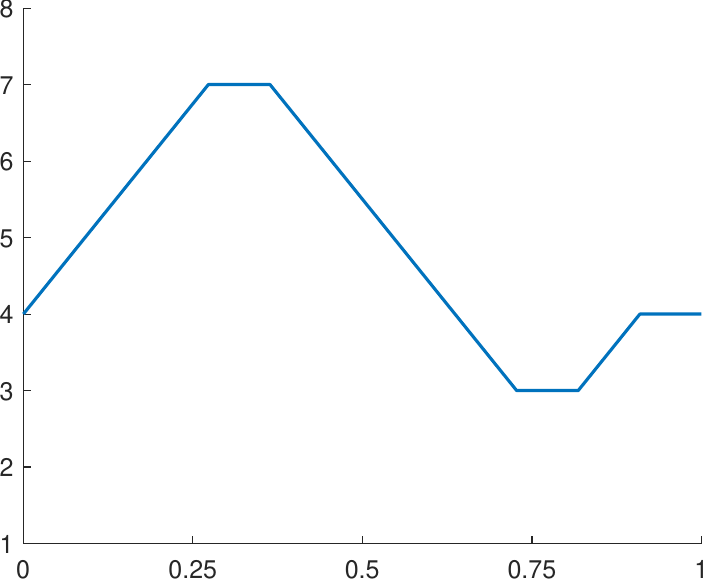}\hspace*{1cm}
		\includegraphics[scale=0.45]{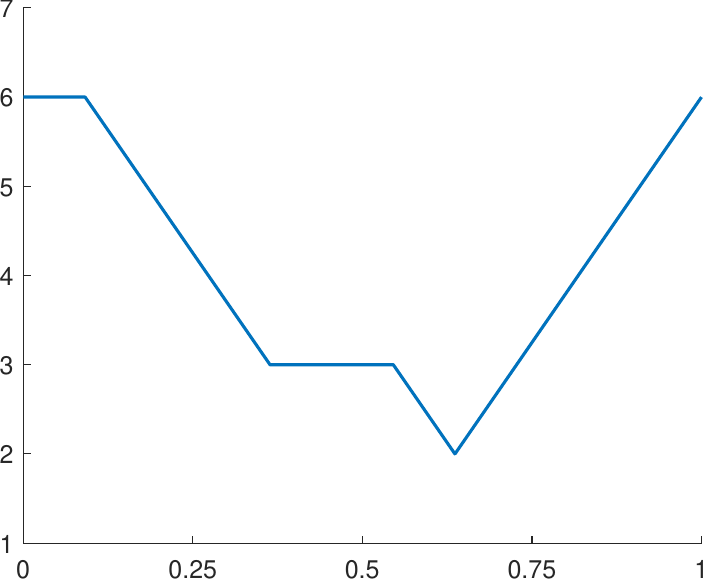}
		 \vspace*{-0.2cm}
	\end{center}
	\caption{The piece-wise {\it linear} functions determined by the abscissas and ordinates of the curve of Figure \ref{ex_percorso} on the left and on the right, respectively. \vspace*{-0.3cm}}
	\label{abs_ord_percorso}
\end{figure}

Next we approximate $x_1$ and $x_2$ by means of an operator $S_n$ for some $n\in \N^+$, i.e. we consider $S_n x_1$ and $S_n x_2$ (we proceed, in particular, as explained in Section \ref{sec_exm} if $S_n$ is not defined for functions with domain $[0,1]$). Finally, putting what we need together, the curve $(S_n x_1,S_n x_2)$ constitutes an approximation of $(x_1,x_2)$ and hence of the image curve after that the discretization is made. 

The entire process is illustrated in Figure \ref{ex_leaf} with an image (shown in (a)) containing a curve with an higher number of points. In particular, following the steps above, we extracted the point coordinates according to the path made by the direction ordering and we defined the abscissas and ordinates functions $x_1$ and $x_2$, which are displayed in blue in Figure \ref{ex_leaf}(c-d). We employed the generalized sampling operator $\mathcal{S}^F_{n}$ with Fejér kernel and $n=100$ to approximate $x_1$ and $x_2$ (the functions $\mathcal{S}^F_n x_1$ and $\mathcal{S}^F_n x_2$ are shown in Figure \ref{ex_leaf}(c-d) in red). 
Figure \ref{ex_leaf}(b) is an image containing the original curve in black and the approximating curve $(\mathcal{S}^F_n x_1,\mathcal{S}^F_n x_2)$ in red after discretization.

\begin{figure}[b]
	\begin{center}
		\includegraphics[scale=0.45]{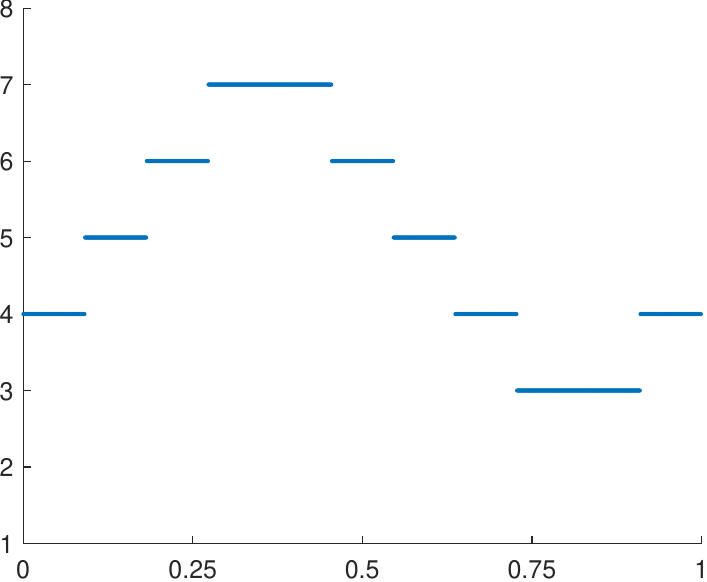}\hspace*{1cm}
		\includegraphics[scale=0.45]{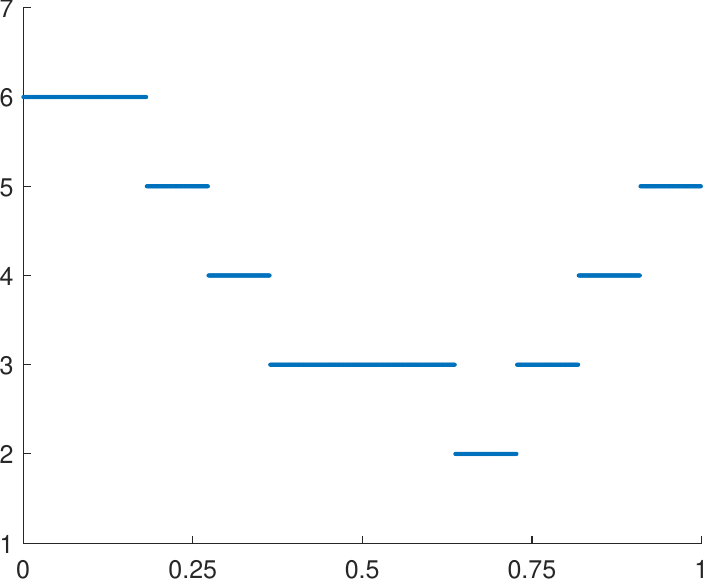}
		 \vspace*{-0.2cm}
	\end{center}
	\caption{The piece-wise {\it constant} functions of abscissas and ordinates of the curve of Figure \ref{ex_percorso} on the left and on the right, respectively.}
	\label{abs_ord_cost_percorso}
\end{figure}

\begin{figure}[h!]
	\begin{center}
		\begin{subfigure}[b]{0.4\textwidth}
			\fbox{\includegraphics[scale=0.75]{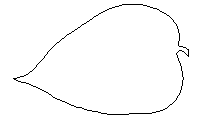}}
			\caption{}
		\end{subfigure}
	\hspace*{1.2cm}
	\begin{subfigure}[b]{0.4\textwidth}
		\fbox{\includegraphics[scale=0.75]{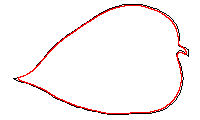}}
		\caption{}
	\end{subfigure}
\end{center}
		\begin{subfigure}[b]{1\textwidth}
			\begin{center}
			\includegraphics[scale=0.3]{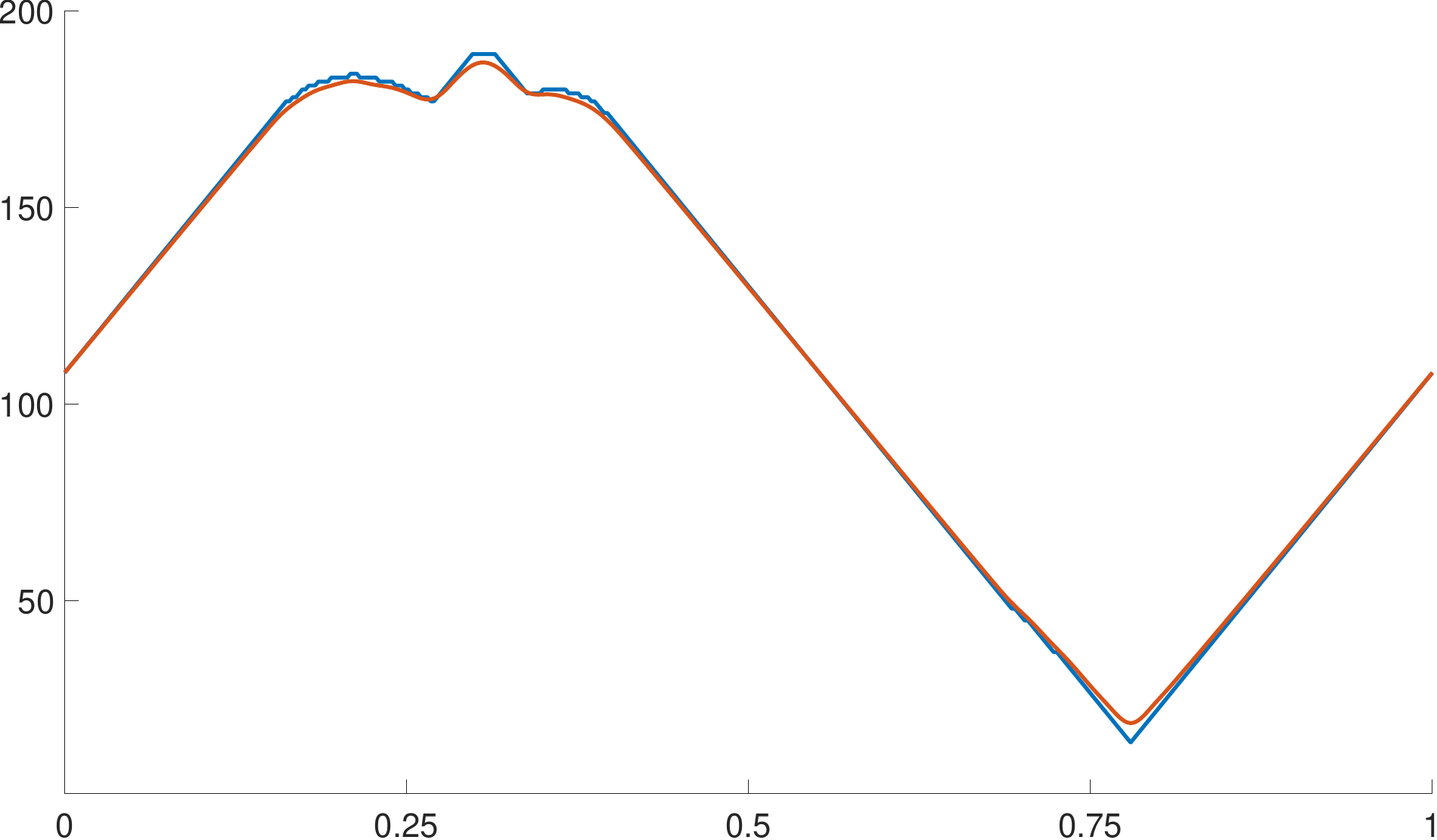}
				\vspace*{-0.2cm}
			\end{center}
			\caption{}
		\end{subfigure}

		\begin{subfigure}[b]{1\textwidth}
			\begin{center}
			\includegraphics[scale=0.3]{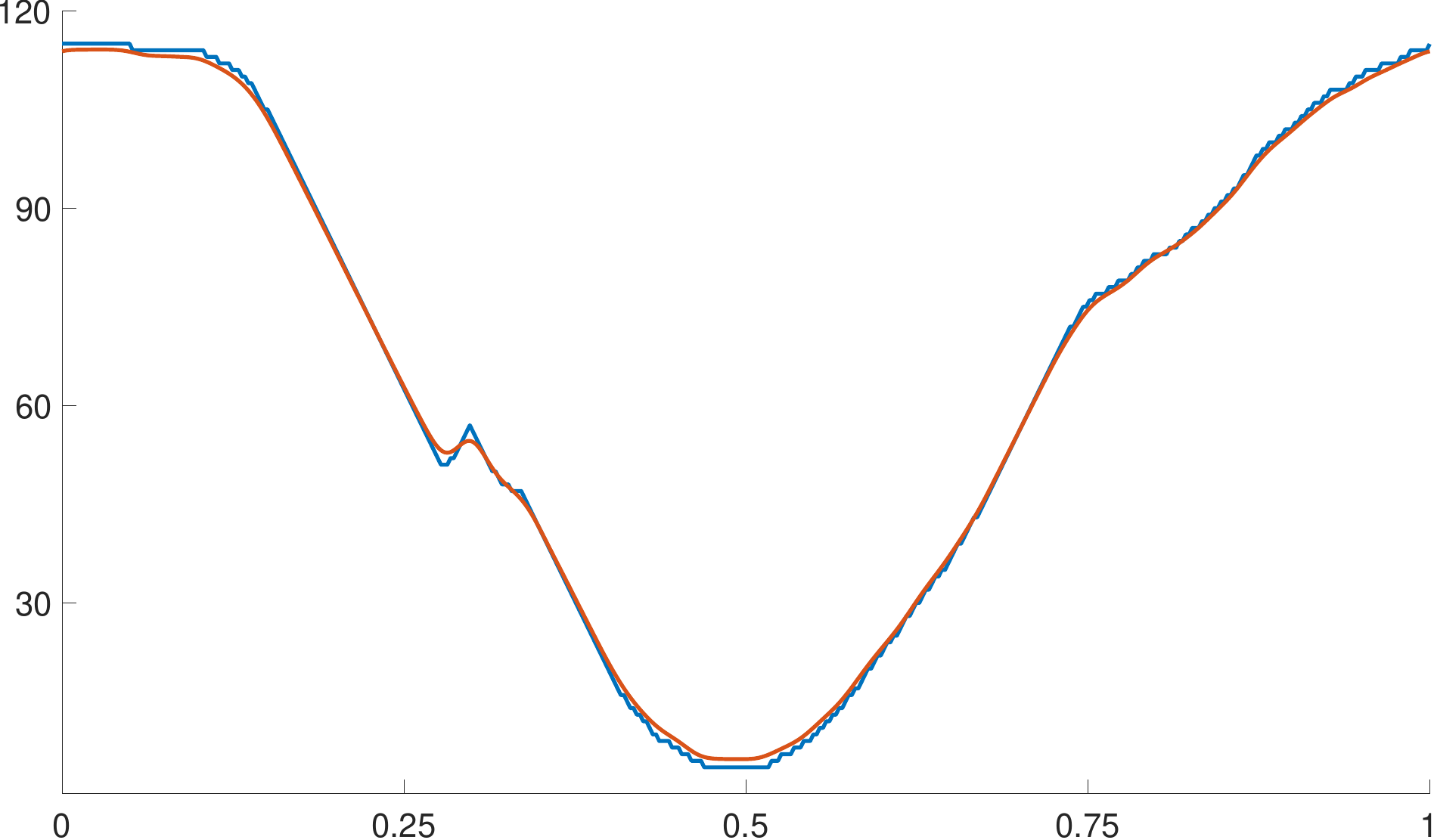}
				\vspace*{-0.2cm}
		\end{center}
			\caption{}
		\end{subfigure}
	
		\caption{An image with a curve is shown in (a) (the resolution of the image is $200\times120$). The functions $x_1$ and $x_2$ giving the abscissas and ordinates of the points of the curve are plotted, in blue, in (c) and (d), respectively. The functions in red in (c) and (d) are, instead, the corresponding approximations $\mathcal{S}^F_{100} x_1$ and $\mathcal{S}^F_{100} x_2$ obtained by the generalized sampling operator $\mathcal{S}^F_{100}$ with Fejér kernel. The discretization of the curve $(\mathcal{S}^F_{100} x_1,\mathcal{S}^F_{100} x_2)$ is shown, in red color, in (b) and is overlaid on the original curve. } 
		\label{ex_leaf}
\end{figure}

\begin{rem}
	A variation of the method consists in defining $x_1$ and $x_2$ as piece-wise constant functions. More precisely, if $u=(u_j)_{j=1}^N$ is the sequence of abscissas of a closed curve, then we define $x_1:[0,1]\to \R$ associated to $u$ as 
	$$
	x_1(t)=u_j \quad\text{if } t\in \left[\frac{j-1}{N},\frac{j}{N} \right[\text{ and } 1\leq j\leq N-1, \qquad x_1(t)=u_N \quad\text{if } t\in\left [\frac{N-1}{N},1 \right]
	$$ 
	and $x_2$ in a similar way in terms of $v=(v_j)_{j=1}^N$ (Figure \ref{abs_ord_cost_percorso} shows the new cases for the same example of Figure \ref{ex_percorso}). 
As before we consider $S_n x_1$ and $S_n x_2$, with $n$ large enough, as approximations of $x_1$ and $x_2$, respectively. Anyway, since in this case $x_1$  and $x_2$ are piece-wise constant, according to Corollary \ref{cor_appr_curves} the approximation holds point-wise except in a finite number of points, of the form $\dsp \frac j N$. Consequently, this happens also for the convergence of $S_n\gamma(t)=((S_nx_1)(t),(S_nx_2)(t))$ to $(x_1(t),x_2(t))$ (note that, even though $x_1$ and $x_2$ are piece-wise constant,  $S_n\gamma:[0,1]\to \R^2$ is still a curve by Corollary \ref{cor_appr_curves}). In particular, the values $\dsp \lim_{n\to +\infty} (S_n x_1)\left (\frac j N\right )$ and $\dsp \lim_{n\to +\infty} (S_n x_2)\left (\frac j N\right )$ may differ from the values of $x_1$ and $x_2$ in a neighborhood of $\frac j N$. By the way, this does not necessarily constitute an issue for the application. 
Indeed, there are some discrete operators $S_n$ with the following properties: if $t_0$ is a jump discontinuity for a function $f$, then $S_n f(t_0)$ converges, as $n\to +\infty$, to an intermediate value between $\dsp \lim_{t\to t_0^+}f(t)$ and $\dsp \lim_{t\to t_0^-}f(t)$ (an example is the generalized sampling operator with some hypothesis about the kernel, see \cite[Theorem 2]{BRS}). Thus, considering these operators and since the jumps of $x_1$ are equal to $1$, the value of $\dsp \lim_{n\to +\infty} (S_n x_1)\left (\frac j N\right )$ after rounding off becomes $u_j$ or $u_{j+1}$. A similar statement holds for $S_nx_2$. In conclusion, defining $x_1$ and $x_2$ as piece-wise constant (and not piece-wise linear) functions, the error between the image curve and the approximation curve for $n$ large enough is, at most, of one pixel of distance. 
\end{rem}

Following the same idea of approximating the coordinates functions, another application can be made: the approximation of a curve from few given points. Let us consider the problem, as in Figure \ref{ex_squirrel_a}, of having only few points of a curve and, in addition, an ordering of them (which follows the path of the original contour). Even though the points are separate, a piece-wise linear curve can be defined connecting the points in the given ordering. The abscissa and ordinate functions for such a possible curve can be still defined as in \eqref{coord_function}. The piece-wise linear curve may give a good representation of the original contour, but it may be at the same time undesired, since it is not smooth. Thus, we can consider an approximation operator $S_n$, so that $(S_n x_1,S_n x_2)$, for $n$ large, determines a curve approximating the piece-wise linear contour and more regular (like in Figure \ref{ex_squirrel_b}). 

\begin{figure}[h!]
	\begin{center}
		\hspace*{1cm}
		\begin{subfigure}{0.45\textwidth}
			\includegraphics[scale=0.32]{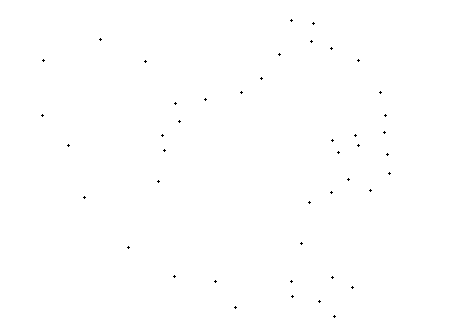}
			\caption{}
			\label{ex_squirrel_a}
		\end{subfigure}
		\begin{subfigure}{0.45\textwidth}
			\includegraphics[scale=0.32]{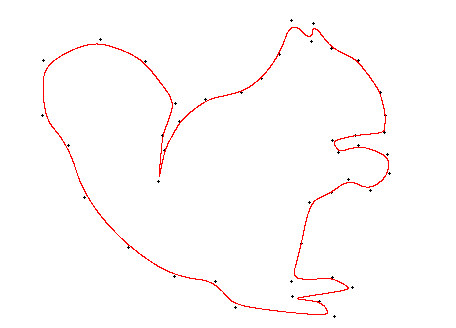}
			\caption{}
			\label{ex_squirrel_b}
		\end{subfigure}
		\caption{Approximation of a curve from a sequence of few points. (a) A set of points (a given ordering is also defined). (b) The same set of points and a curve created by applying a generalized sampling operator. } 
		\label{ex_squirrel}
	\end{center}
\end{figure}

\subsection{Reconstruction of curves in images}

The technique described in the previous section leads to move from a discrete object (the curve in an image) to a continuous object (the curve $S_n \gamma$). This allows to make some operations on the continuous object that give better results in comparison to working with the discrete initial object. For example, once we approximate $\gamma$ with $S_n \gamma$, we can dilate $S_n \gamma$ and then convert it to a curve in an image by discretization. The result is a larger image which is not a simple scaling of the original image. Speaking in the imaging language, the operator $S_n$ can be also used to reconstruct the curves in images with an higher resolution. 

In Figure \ref{ex_leaf_resol} we show the results of this procedure. The curve $\gamma$ of the same image case of Figure \ref{ex_leaf} is shown in Figure \ref{ex_leaf_resol_orig}. The resolution of the image is $200\times 120$. The image scaled by a factor equals to $2$ is displayed in Figure \ref{ex_leaf_resol_dil} (the effect is that of a zoom, the resolution is not changed). The generalized sampling operator $\mathcal{S}^{M_3}_{n}$ with B-spline kernel of order $3$ and $n=200$ was applied to approximate $\gamma$ as explained in the previous section. The obtained curve $\mathcal{S}^{M_3}_{n} \gamma$ is multiplied by $2$ and then discretized (i.e., the values of the coordinates are rounded off towards the nearest integer). 
The final result is in Figure \ref{ex_leaf_resol_ris} and, by effect of the multiplication, the curve is contained in an image with increased resolution ($400 \times 240$). Note the difference of quality of the images from a comparison between Figures \ref{ex_leaf_resol_dil} and \ref{ex_leaf_resol_ris}.

\begin{figure}[h!]
	\begin{center}
		\begin{subfigure}{0.4\textwidth}
			\hspace*{1cm}
			\includegraphics[scale=0.5]{Leaf.png}
			\caption{}
			\label{ex_leaf_resol_orig}
		\end{subfigure}\\
	\end{center}
	\begin{center}
	\begin{subfigure}{0.47\textwidth}
		\includegraphics[scale=1]{Leaf.png}
		\caption{}
		\label{ex_leaf_resol_dil}
	\end{subfigure}
		\begin{subfigure}{0.47\textwidth}
			\includegraphics[scale=0.5]{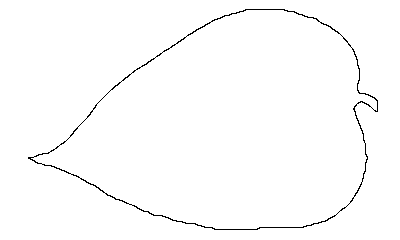}
			\caption{}
			\label{ex_leaf_resol_ris}
		\end{subfigure}
		\caption{In (a): the curve in a image. In (b): the image in (a) with doubled scale. In (c): the curve in (a) reconstructed by a generalized sampling operator and dilated (the image has the same scale of (a) but doubled resolution). } 
		\label{ex_leaf_resol}
	\end{center}
\end{figure}
\vspace*{-0.5cm}
Other operations that can be performed after the approximating curve $S_n \gamma$ is obtained are, for examples, translations, rotations or dilations along specific directions, projections. Working on an auxiliary curve of the real plane as $S_n \gamma$, defined by a continuous and not discrete variable, allows to make operations with better results.

\bigskip
\bigskip
\bigskip
{\bf Authors contribution:} 
Rosario Corso: conceptualization, formal analysis, investigation, methodology, software, supervision, visualization, writing. Gabriele Gucciardi: formal analysis, investigation, software.

\bigskip
{\bf Acknowledgments:} 
R.C. was partially supported by the ``Gruppo Nazionale per l'Analisi Matematica, la Probabilità e le loro Applicazioni'' (INdAM) and this work has been done within the activities of the ``Gruppo UMI - Teoria dell’Approssimazione e Applicazioni''.

\end{document}